\documentclass[a4paper]{article}
\usepackage{amsmath,amsthm,amssymb}
\usepackage[all]{xy}
\usepackage{todonotes}

\newcommand{\pushoutcorner}[1][dr]{\save*!/#1+1.2pc/#1:(1,-1)@^{|-}\restore}
\newcommand{\pullbackcorner}[1][dr]{\save*!/#1-1.2pc/#1:(-1,1)@^{|-}\restore}

\newtheorem{theorem}{Theorem}[section]
\newtheorem{proposition}[theorem]{Proposition}
\newtheorem{lemma}[theorem]{Lemma}

\newtheorem{remark}[theorem]{Remark}

\theoremstyle{definition}
\newtheorem{axiom}[theorem]{Axiom}

\newtheorem{definition}[theorem]{Definition}

\newcommand{\sub}{\mathsf{01Sub}}

\newcommand{\monadf}{}
\newcommand{\btwo}{\mathbf{2}}
\newcommand{\bthree}{\mathbf{3}}

\newcommand{\fmap}{\monadf{F}\text{-}\mathbf{Map}}
\newcommand{\rmap}{\monadf{R}\text{-}\mathbf{Map}}
\newcommand{\lMap}{\monadf{L}\text{-}\mathbf{Map}}
\newcommand{\lmap}{\monadf{L}\text{-}\mathbf{map}}
\newcommand{\ctmap}{\monadf{C}^t\text{-}\mathbf{map}}
\newcommand{\rtmap}{\monadf{F}^t\text{-}\mathbf{Map}}
\newcommand{\cmap}{\monadf{C}\text{-}\mathbf{map}}
\newcommand{\cMap}{\monadf{C}\text{-}\mathbf{Map}}
\newcommand{\wmap}{\monadf{W}\text{-}\mathbf{Map}}

\newcommand{\catc}{\mathbb{C}}

\newcommand{\poprod}{\hat{\otimes}}

\newcommand{\id}{\operatorname{Id}}

\newcommand{\intv}{\mathbb{I}}

\newcommand{\monexpl}[2]{\hat \hom(#1 , #2)}

\title{Identity Types in Algebraic Model Structures and Cubical Sets}
\author{Andrew W Swan}

\begin{document}

\maketitle

\begin{abstract}
  We give a general technique for constructing a functorial choice of
  very good paths objects, which can be used to implement identity
  types in models of type theories in direct manner with little
  reliance on general coherence results. We give a simple proof that
  applies in algebraic model structures that possess a notion of
  structured weak equivalence, in a sense that we define here. We then
  give a more direct proof that applies both to the original BCH
  cubical set model and more recent variants. We give an explanation
  how this construction relates to the one used in the CCHM cubical
  set model of type theory.
\end{abstract}

\section{Introduction}
\label{sec:introduction}

In the original Bezem-Coquand-Huber cubical set model of type theory
in \cite{bchcubicalsets}, Bezem, Coquand and Huber only showed the
$J$-computation rule for identity types held only up to propositional
equality rather than the more usual definitional equality. In
\cite{swannomawfs} the author gave both an explanation for why this
was the case and a solution. The explanation was a Brouwerian counterexample
based on the nerve of a complete metric space, demonstrating that
there is no constructive proof that the necessary strict equality in
the $J$-computation rule holds. The solution was to consider a second
more elaborate definition of identity type.

The motivation for the work here is to give a more conceptually clear
proof of this construction, by viewing it in terms of the
cofibration-trivial fibration factorisation of an algebraic model
structure. From this point of view the basis of the construction is
a functorial version of
the following simple and well known trick. Suppose we are given a
fibration $f \colon X \to Y$ in a model structure. A path object on
$f$ is a factorisation of the diagonal map
$\Delta_f \colon X \to X \times_Y X$ as a weak equivalence
$r \colon X \to P f$ followed by a fibration
$p \colon P f \to X \times_Y X$. A very good path object is the same,
but where $r$ is required to be a trivial cofibration, not just a weak
equivalence. If we are given a path object, then we can use it to
produce a very good path object, by factorising the weak equivalence
$r$ as a cofibration $C r$ followed by a trivial fibration $F^t r$, as
in the diagram below.
\begin{equation*}
  \begin{gathered}
    \xymatrix{ X \ar[rr]^\Delta \ar[dr]_r \ar[d]_{C r} & & X \times_Y X \\
      M r \ar[r]_{F^t r} & P f \ar[ur]_p & }
  \end{gathered}
\end{equation*}
By the $3$-for-$2$ property $C r$ is a weak equivalence, and so a
trivial cofibration, and $p \circ F^t r$ is fibration, making $M r$ a
very good path object.

We will see a more complete proof (in comparison to
\cite{swannomawfs}) that the resulting structure can be used to
implement identity types. For this we will use a notion of stable
functorial very good path object due to Van den Berg and Garner
\cite{vdberggarnerpathobj}. This has the advantage that it can be used
to implement identity types in a relatively direct way with less
reliance on general coherence results (the advantages of this approach
will be discussed further in section
\ref{sec:identity-types-an-1}). However in order to satisfy this
definition it is no longer sufficient to work in a model
structure. Instead we build on the notion of \emph{algebraic model
  structure} due to Riehl \cite{riehlams}. We expand on Riehl's
definition by adding a structured notion of weak equivalence and show
how to use such structures to construct stable functorial choices of
very good path object. We then adapt the straightforward proof above
to a functorial version using this definition.

It is however non trivial to satisfy this definition of algebraic
model structure with structured weak equivalence. Because of this, we
will also give a more direct proof that the same construction yields
identity types in BCH cubical sets, and many other categories.

In a separate paper the author will prove that BCH cubical sets do in
fact satisfy this definition of algebraic model structure with
structured weak equivalences. This will use a result due to Sattler
\cite{sattlermodelstructures}, together with Huber's proof that the
universe of small types is fibrant \cite{huberthesis}, combined with
an argument using Grothendieck fibrations and other
observations. Moreover, further results by Sattler (currently
unpublished) suggest this
extra structure can be found in a wide variety of categories,
including CCHM cubical sets.

An earlier draft of this work was circulated online, and for reference
remains available at \cite{swanidamsdraft}. The main change is that
the original draft contained some rather messy arguments based on the
concrete definition of the relevant awfs's in $01$-substitution
sets. In contrast, the results here will use an approach developed by
Gambino and Sattler in \cite{gambinosattlerpi} (and following some
suggestions by Peter Lumsdaine). This yields a much
clearer proof and much more general result, although essentially
following the same outline as the original. The definition of ams with
structured weak equivalences has been slightly generalised, but as we
will see it turns out to not be essentially different to the original
version. The earlier draft also contained an unproved claim that
$01$-substitution sets are an algebraic model structure with
structured weak equivalences, which as mentioned above will now appear
in a separate paper.

The construction of identity types in the
Cohen-Coquand-Huber-M\"{o}rtberg (CCHM) cubical set model
\cite{coquandcubicaltt} was inspired by this work via the earlier
draft and correspondence with the author. This included some
simplifications discovered by Cohen, Coquand, Huber and M\"{o}rtberg
that apply to CCHM cubical
sets. This definition was generalised to a large class of models by
Orton and Pitts in \cite{pittsortoncubtopos} and by Van den Berg and
Frumin in \cite{vdbergfrumin}. To be clear however, these
simplifications do not always apply, and for instance do not include
the original BCH cubical set model, which will be covered by the
results in this paper. In section \ref{sec:simpl-when-step} we will
give an abstract view of the the CCHM definition and see how it
relates to the definition given here.

\subsection*{Acknowledgements}
\label{sec:acknowledgements}

I'm grateful for useful discussions, suggestions, corrections, etc
from Benno van den Berg, Martijn den Besten,
John Bourke, Thierry Coquand, Nicola Gambino,
Simon Huber, Peter Lumsdaine, Ian Orton, Andy Pitts, Emily Riehl and
Christian Sattler.

A large part of this work was carried out while at the
University of Leeds, under EPSRC grant no. EP/K023128/1.

\section{Review of Algebraic Weak Factorisation Systems and Model Structures}
\label{sec:background}

\begin{definition}
  \label{def:8}
  Let $\catc$ be a category and let $i : U \rightarrow V$ and $f : X
  \rightarrow Y$ be morphisms in $\catc$. We write $i \pitchfork f$
  and say $i$ has the left lifting property with respect to $f$ and
  $f$ has the right lifting property with respect to $i$ if the
  following holds.
  For every commutative square of the form
  \begin{equation*}
    \begin{gathered}
      \xymatrix{ U \ar[r] \ar[d]_i & X \ar[d]_f \\
        V \ar[r] & Y}
    \end{gathered}
  \end{equation*}
  there is a diagonal map $j$ as below, making the two triangles
  commute.
  \begin{equation*}
    \begin{gathered}
      \xymatrix{ U \ar[r] \ar[d]_i & X \ar[d]_f \\
        V \ar[r] \ar@{.>}[ur]|j & Y}
    \end{gathered}
  \end{equation*}
\end{definition}

\begin{definition}
  \label{def:7}
  Let $\catc$ be a category and $\mathcal{M}$ a class of maps in
  $\catc$. We define
  \begin{align}
    \mathcal{M}^\pitchfork &= \{ f \;|\; (\forall i \in
    \mathcal{M})\,i \pitchfork f \} \\
    {}^\pitchfork \mathcal{M} &= \{ i \;|\; (\forall f \in
    \mathcal{M})\,i \pitchfork f \}    
  \end{align}
\end{definition}

\begin{definition}
  \label{def:6}
  Let $\catc$ be a category. A \emph{weak factorisation system} on
  $\catc$ consists of classes of maps $\mathcal{C}$ and $\mathcal{F}$
  such that $\mathcal{C} = {}^\pitchfork \mathcal{F}$ and $\mathcal{F}
  = \mathcal{C}^\pitchfork$ and every morphism $g$ in $\catc$ factors
  as $g = f \circ i$ with $i \in \mathcal{C}$ and $f \in
  \mathcal{F}$.
\end{definition}

\begin{definition}[Quillen]
  Let $\catc$ be a category. A \emph{model structure} on $\catc$
  consists of two weak factorisation systems $(\mathcal{C},
  \mathcal{F})$ and $(\mathcal{C}^t, \mathcal{F})$, together with a
  class of morphisms $\mathcal{W}$ such that the following hold.
  \begin{enumerate}
  \item $\mathcal{C}^t = \mathcal{C} \cap \mathcal{W}$
  \item $\mathcal{F}^t = \mathcal{F} \cap \mathcal{W}$
  \item ($3$-for-$2$) If $f \colon X \to Y$, $g \colon Y \to Z$ and $h
    := g \circ f$, and any two maps out of $f$, $g$ and $h$ belong to
    $\mathcal{W}$, then so does the third.
  \end{enumerate}
\end{definition}

\begin{definition}
  \label{def:9}
  Let $\catc$ be a category. A \emph{functorial factorisation} on
  $\catc$ consists of a functor $\catc^\btwo \rightarrow
  \catc^\bthree$ which is a section to the composition functor
  $\catc^\bthree \rightarrow \catc^\btwo$. We will usually write out a
  functorial factorisation as three separate components $L, K,
  R$ as follows. E.g., if $f$ is an object of $\catc^\btwo$ (i.e. a
  morphism in $\catc$) we might write the factorisation as
  \begin{equation*}
    \begin{gathered}
      \xymatrix{ X \ar[rr]^f \ar[dr]_{L f} & & Y \\
        & K f \ar[ur]_{R f} & }
    \end{gathered}
  \end{equation*}
\end{definition}

\begin{definition}[Grandis, Tholen]
  \label{def:3}
  Let $\catc$ be a category and $(L, R)$ a functorial factorisation
  on $\catc$. Note that $L$ is an endofunctor on $\catc^\btwo$ and can
  be made into a copointed endofunctor in a canonical way. Dually, $R$
  can be made into a pointed endofunctor. An \emph{algebraic weak
    factorisation system} on $\catc$ consists of a functorial
  factorisation together with a comultiplication map
  $\Sigma : L \rightarrow L^2$ making $L$ into a comonad and a
  multiplication map $\Pi : R^2 \rightarrow R$ making $R$ into a
  monad. Furthermore, the canonical map $L R \rightarrow R L$ is a
  distributive law.
\end{definition}

We will write the category of coalgebras over the comonad as $\lMap$
and the category of algebras over the monad as $\rmap$. In many cases
it is difficult (or impossible) to show that a map satisfies the
comultiplication law required to be an $L$-coalgebra. For this reason
we will usually work with the category of coalgebras over the underlying
copointed endofunctor of $L$. We will write this category as
$\lmap$. For example $\lmap$ is always closed under retracts whereas
$\lMap$ is not.

The dual issue for $R$-algebras does not cause problems in practice,
because of the following proposition.

\begin{proposition}
  Suppose that $(L, R)$ is a cofibrantly generated awfs on a category
  $\catc$. Then given a algebra structure for the underlying pointed
  endofunctor for $R$ on a map $f$ in $\catc$, we can functorially
  assign $f$ the structure of an algebra over the monad $R$.
\end{proposition}

\begin{proof}
  See e.g. \cite[Lemma 2.30]{riehlams}
\end{proof}

For $R$, we will only every work over the category of (monad)
algebras, $\rmap$. The reason is that for cubical sets (as in
\cite{bchcubicalsets} and \cite{coquandcubicaltt}) the definition of Kan
fibration is already fixed and used in the interpretation of type
theory, and in practice $\rmap$ corresponds more closely to these
definitions.

There is also a special case where we can do the same for left maps,
as we'll see in proposition \ref{prop:fixcofibs}.

Algebraic model structures were developed by Riehl in \cite{riehlams}.
Before giving Riehl's definition, we first define a weaker version
that will play an important role.

\begin{definition}
  A \emph{pre algebraic model structure} (pre-ams) consists of two
  awfs's $(C^t, F)$ and $(C, F^t)$ together with a morphism of awfs's
  $\xi \colon (C^t, F) \rightarrow (C, F^t)$. We refer to the morphism
  $\xi$ as the \emph{comparison map} of the pre-ams.
\end{definition}

\begin{definition}[Riehl]
  An \emph{algebraic model structure} consists of a pre algebraic
  model structure on a complete and cocomplete category together with
  a class of maps $\mathcal{W}$ such that if $\mathcal{C}$ is the
  class of maps that admit a (copointed endofunctor) $C$-coalgebra
  structure, $\mathcal{F}$ is the class of maps that admit a (pointed
  endofunctor) $F$-algebra structure, then
  $(\mathcal{C}, \mathcal{F}, \mathcal{W})$ is a model structure.
\end{definition}

\section{Identity Types in an Awfs}
\label{sec:identity-types-an-1}

In this section we will review a description of the semantics for
identity types due to Van den Berg and Garner \cite[Section
3]{vdberggarnerpathobj}. We first talk about two related well known
issues with the implementation of identity types that arise in the
description of identity types by Awodey and Warren, and are elegantly
resolved by the Van den Berg-Garner definition.

\subsection{Coherence for $J$-Terms}
\label{sec:coherence-j-terms}

It is a well known issue in type theory that great care needs to taken
to ensure the interpretation of type theory into categorical semantics
is correct, due to so called coherence issues. This was noticed and
then solved by Hofmann in \cite{hofmannlcc} for the interpretation of
extensional type theory into a locally cartesian closed
category. Essentially the issue is as follows.

In most formulations of models of type theory, such as categories with
families (CwFs), one needs to have a notion of substitution for types
and terms, and furthermore the substitution needs to be strict in the
following sense. If we are given morphisms of contexts
$\sigma \colon \Xi \rightarrow \Delta$ and
$\tau \colon \Delta \rightarrow \Gamma$, and a type $\Gamma \vdash X$,
then we need to ensure that $X[\tau][\sigma]$ is strictly equal to
$X[\tau \circ \sigma]$ (as types in context $\Xi$). Similarly for
terms.

For the interpretation of extensional type theory in a locally
cartesian closed category, this is an issue for interpreting types,
but the interpretation of terms is not a problem.

On the other hand when Awodey and Warren developed the interpretation
of identity types using very good path objects in \cite{awodeywarren},
there is a coherence issue for terms. Specifically, as Awodey and
Warren explain in \cite[Section 4.1]{awodeywarren}, they do not ensure
that $J$ terms are preserved by substitution. The reason is that $J$
terms are interpreted as diagonal fillers of certain lifting
problems. In a weak factorisation system, we are only guaranteed that
at least one filler exists for every lifting problem of a trivial
cofibration against a fibration. So there's no reason to expect the
different choices of fillers to agree with each other under
substitution.

It is possible to fix this issue using general coherence theorems,
such as local universes (developed by Lumsdaine and Warren in
\cite{lumsdainewarren}) or using a universe of small types, as used
for simplicial sets in \cite{voevodskykapulkinlumsdainess}. However,
the Van den Berg-Garner interpretation allows us to
deal with this problem in a much more direct way, with less reliance
on general coherence results, and which can be used directly in an
existing CwF with less work.

\subsection{The Computational Meaning of Transport}
\label{sec:comp-mean-transp}

In more computationally minded approaches to homotopy type theory,
there is an emphasise on the computational meaning of transport, which
is in turn is strongly connected with the $J$ terms.

For example, this was noticed early on by Harper and Licata in
\cite{harperlicata}, and by Bezem, Coquand and Huber in
\cite{bchcubicalsets}, but also plays an important part in more recent
developments.

The issue is as follows. Suppose we are given a type $A$ in context
$\Gamma$. For simplicity, say that $\Gamma$ consists of a single type
$C$. Suppose further we are given $c_0, c_1 : C$ and also a term $p$
of type $\id_C(c_0, c_1)$. Then we need to show how we can take a term
$a$ of type $A(c_0)$ and then compute a new term of type
$A(c_1)$. Using univalence we show that there are non trivial
instances of this problem. We take $C$ to be the universe of small
types, $A(x)$ is defined to be the type $\operatorname{El}(x)$, and
$p$ is an identity between two types constructed from an equivalence
using univalence. Then in order to compute what the transport should
be, we need to recover the computational information from the
equivalence that we put in.

The solution is that whenever we interpret a type $\Gamma \vdash A$ in
a model, we include all the computational information we need about
how to compute transport in $A$ for paths in $\Gamma$. The Van den
Berg-Garner approach allows us to clearly see the necessary structure
in an abstract way which is conceptually very similar to the
Awodey-Warren approach. Namely, we work in a setting where it is
natural to view fibrations not as a class of maps but as algebras over
a monad. The computational information we need to associate with a
type is precisely contained in an algebra structure over the monad.

\subsection{The Van den Berg-Garner Interpretation of Identity Types}
\label{sec:van-den-berg}

The key part of the Van den Berg-Garner interpretation is that instead
of a weak factorisation system, they use an \emph{algebraic weak
  factorisation system}\footnote{The exact formulation used
  by Van den Berg and Garner is not quite an awfs, but a slightly
  weaker notion.}.

With an awfs $(L, R)$, it is natural to view fibrations not as just a
class of maps (as is the case for wfs's) but instead as a category of
algebras over the monad $R$. Meanwhile trivial cofibrations are best
viewed as coalgebras. Given a map $m$ together with coalgebra structure
and a map $f$ with algebra structure, $f$ not only has the right
lifting property against $m$, but using the structures, we have a
choice of diagonal filler. Furthermore, given a morphism of coalgebras
and a morphism of algebras, we also get compatible diagonal fillers.

Then, as Van den Berg and Garner show in \cite[Section
2]{vdberggarnerpathobj}, we can use this to define a type category
where types are implemented as $R$-algebras, and substitution is
implemented as pullbacks that preserve $R$-algebra structure. Then the
algebra structure contains the computational information that we need
to implement transport.

Finally, Van den Berg and Garner implement identity types in the type
category using the following definition\footnote{We change the
  terminology to fit better with other ideas in this paper.}. Observe
that we require not just that certain maps are fibrations
and trivial cofibrations, but that we are given a choice of algebra
and coalgebra structure. This allows one to give an explicit definition
of diagonal fillers and thereby of the interpretation of $J$-terms.
Furthermore, we require functoriality with respect to trivial cofibrations and
fibrations everywhere. This ensures that the choice of diagonal
fillers, and so of $J$ terms is stable under substitution.

\begin{definition}[Van den Berg, Garner]
  \label{def:strpathobj}
  \begin{enumerate}
  \item A \emph{choice of very good path objects} consists of an
    assignment to every $F$-map $f : X \rightarrow Y$ a factorisation
    \begin{equation}
      \label{eq:7a}
      X \stackrel{r_f}{\rightarrow} P(f) \stackrel{p_f}{\rightarrow} X
      \times_Y X
    \end{equation}
    of the diagonal $\Delta : X \rightarrow X \times_Y X$ together
    with an $C^t$-coalgebra structure on $r_f$ and an $F$-algebra
    structure on $p_f$.
  \item A choice of very good path objects is \emph{functorial} if
    the assignment of \eqref{eq:7a} provides the action of objects of a
    functor $\fmap \rightarrow \fmap \times_{\catc} \ctmap$.
  \item A choice of very good path objects is \emph{stable} when
    every map of $F$-algebras whose underlying square is a pullback
    makes the following square given by functoriality a pullback
    \begin{equation*}
      \begin{gathered}
        \xymatrix{P(f) \ar[r]^{P(h,k)} \ar[d]_{p_f} & P(f')
          \ar[d]_{p_{f'}} \\
          X \times_Y X \ar[r] & X' \times_{Y'} X'}
      \end{gathered}
    \end{equation*}
  \item The awfs is \emph{Frobenius} if to every square
    \begin{equation*}
      \begin{gathered}
        \xymatrix{ f^\ast X \ar[r]^{\bar f} \ar[d]_{\bar i} & X
          \ar[d]_i \\
          Z \ar[r]^f & Y}
      \end{gathered}
    \end{equation*}
    together with an $F$-algebra structure on $f$ and $C^t$-coalgebra
    structure on $i$ we have assigned an $C^t$-coalgebra structure on
    $\bar i$. It is \emph{functorially Frobenius} if this assignment
    gives rise to a functor $\fmap \times_{\catc} \ctmap \rightarrow
    \ctmap$.
  \item A \emph{homotopy theoretic model of identity types} is a
    finitely complete category $\catc$ together with an awfs that is
    functorially Frobenius and has a stable functorial choice of
    very good path objects.
  \end{enumerate}
\end{definition}

\begin{theorem}[Van den Berg, Garner]
  \label{thm:htmidtsound0}
  Every homotopy theoretic model of identity types gives rise to a
  model of type theory with identity types.
\end{theorem}

\begin{proof}
  See \cite[Section 3.3]{vdberggarnerpathobj}.
\end{proof}

\begin{remark}
  We leave it for future work to check that theorem
  \ref{thm:htmidtsound0} can be proved constructively.
\end{remark}

\section{Identity Types in an Algebraic Model Structure}
\label{sec:identity-types-an}

In this section we give our first construction of a stable functorial
choice of very good path objects. This is an intuitively clear proof,
based on a simple trick that can be carried out in any model
structure. However, showing that BCH cubical sets satisfy the
conditions needed in order to apply the theorem is difficult, and in
fact will appear in a separate paper. We include the
theorem here anyway, since it illustrates the main idea which will be
used in section \ref{sec:some-suff-cond}, where the results can be
easily applied to cubical sets.

\subsection{Structured Weak Equivalences}
\label{sec:functorial-3-2}

\begin{definition}
  Suppose we are given a pre-ams $\xi : (C^t, F) \rightarrow (C, F^t)$
  on a category $\catc$.

  Suppose further we are given a category $\wmap$ together with a
  faithful functor $\wmap \to \catc^\btwo$. We refer to objects in
  $\wmap$ as \emph{structured weak equivalences}. Given $f$ in
  $\catc^\btwo$, we say a \emph{weak equivalence structure on $f$} is
  an object in the preimage of $f$. Given objects $f$ and $g$ of
  $\catc^\btwo$ together with weak equivalence structures on $f$ and
  $g$, we say a morphism in $\catc^\btwo$ (i.e. a commutative square)
  is a \emph{morphism of weak equivalences} if it is the image of
  a morphism in $\wmap$ between structured weak equivalences.
\end{definition}



\begin{definition}
  \label{def:13}
  Suppose we are given a pre-ams $\xi : (C^t, F) \rightarrow (C, F^t)$
  on a category $\catc$ together with a faithful functor $\wmap \to
  \catc^\btwo$. A \emph{functorial 3-for-2 operator} is the
  following. Given morphisms $f_1, f_2, f_3$ such that
  $f_3 = f_2 \circ f_1$ and suppose for $i \neq j \in \{1,2,3\}$ we are given
  weak equivalence structures on $f_i$ and $f_j$, then writing $k$ for
  the remaining element of $\{1,2,3\}$ we have assigned a weak
  equivalence structure on $f_k$. Furthermore, these assignments are
  functorial, in the following sense. Suppose we are given a
  commutative diagram as below.
  \begin{equation*}
    \begin{gathered}
      \xymatrix{ U \ar[r]^{f_1} \ar[d] & V \ar[r]^{f_2} \ar[d] & W
        \ar[d] \\
        X \ar[r]^{g_1} & Y \ar[r]^{g_2} & Z}
    \end{gathered}
  \end{equation*}
  Let $f_3 := f_2 \circ f_1$ and $g_3 := g_2 \circ g_1$ and write
  $\alpha_1$ for the left hand square $\alpha_2$ for the right hand
  square and $\alpha_3$ for the big rectangle. If $i \neq j \in
  \{1,2,3\}$ and we are given weak equivalence structures on $f_i,
  g_i, f_j$ and $g_j$ such that $\alpha_i$ and $\alpha_j$ are
  morphisms of structured weak equivalences, then $\alpha_k$ is a
  morphism between the weak equivalence structures we have assigned on
  $f_k$ and $g_k$.
\end{definition}

\begin{definition}
  \label{def:14}
  An \emph{ams with structured weak equivalences} is a pre-ams on a
  finitely complete and cocomplete category $\catc$ together with a
  category $\wmap$, a faithful functor $\wmap \to \catc$ and the
  following:
  \begin{enumerate}
  \item A functorial 3-for-2 operator.
  \item Given a weak equivalence structure and a $C$-coalgebra structure
    on each map $f$ a choice of $C^t$-coalgebra structure on $f$ which
    is the action on objects of a functor $\cmap \times_{\catc^\btwo}
    \wmap \rightarrow \ctmap$.
  \item Given a weak equivalence structure and a $F$-algebra structure
    on a map $f$ a choice of $F^t$-algebra structure which is the
    action on objects of a functor $\fmap \times_{\catc^\btwo} \wmap
    \rightarrow \rtmap$.
  \item Given a $C^t$-coalgebra structure on each map $f$, a choice of
    weak equivalence structure on $f$ which is the action on objects
    of a functor $\ctmap \rightarrow \wmap$.
  \item Given an $F^t$-algebra structure on each map $f$, a choice of
    weak equivalence structure on $f$ which is the action on objects
    of a functor $\rtmap \rightarrow \wmap$.
  \end{enumerate}
\end{definition}

\begin{remark}
  We don't assume the existence of any additional structure on
  $\wmap$. Recently Bourke has shown in \cite{bourkewealg} that in
  some natural situations weak equivalences can be viewed as algebras
  over a monad. However, he also showed that this is not the case
  e.g. for simplicial sets with the Kan model structure.
\end{remark}

\subsection{Some Remarks on Ams's with Structured Weak Equivalences}

\subsubsection{Pointwise Ams's with Structured Weak Equivalences}
\label{sec:pointwise-amss-with}

Given an ordinary model structure on a category $\catc$ there are
three main ways to define a model structure on a functor category
$\catc^\mathcal{A}$, under suitable conditions. These are
the projective model structure, injective model structure and the
Reedy model structure.

We will see however (following a suggestion by Emily Riehl) that there
is another option inherent in the definition of ams with structured
weak equivalence, that is different to the standard constructions for
ordinary model structures. This construction does not require any
additional assumptions on the underlying ams with structured weak
equivalences, or on $\mathcal{A}$ or $\catc$.

Suppose that we are given an ams with structured weak equivalences on
a category $\catc$, and another category $\mathcal{A}$. We will define a
new ams with structured weak equivalences on the functor category
$\catc^\mathcal{A}$.

First recall that given an awfs on $\catc$ we can define the
\emph{pointwise awfs} on $\catc^\mathcal{A}$.

\begin{definition}
  Suppose we are given an awfs $(L, R)$ on a category $\catc$. The
  \emph{pointwise awfs} on $\catc^\mathcal{A}$ is the awfs
  $(L_\mathcal{A}, R_\mathcal{A})$ defined as follows. Note that we
  need to define functors $(\catc^\mathcal{A})^\btwo \to
  (\catc^\mathcal{A})^\btwo$, however it suffices to instead define
  functors $\mathcal{A} \times (\catc^\btwo)^\mathcal{A} \to
  \catc^\btwo$. We define $L_\mathcal{A}$ by composition of $L$ with
  the evaluation map $\mathcal{A} \times (\catc^\btwo)^\mathcal{A} \to
  \catc^\btwo$, and define $R_\mathcal{A}$ by composition of $R$ with
  evaluation. We similarly define multiplication and comultiplication
  pointwise. Namely, $\mu$ needs to be a natural transformation
  $R_\mathcal{A}^2 \to R_\mathcal{A}$. Hence for each $f \in
  (\catc^\btwo)^\mathcal{A}$, we need $\mu_f$ to be a natural
  transformation from $R_\mathcal{A}^2(f) \to
  R_\mathcal{A}(f)$. Hence, for each $A \in \mathcal{A}$, we need
  a map $\mu_{f, A} \colon R^2(f(A)) \to R(f(A))$. We take $\mu_{f,
    A}$ to be $\mu_{f(A)}$. Comultiplication is defined
  similarly. Naturality and the other required equalities follow from
  the corresponding conditions on $(L, R)$.
\end{definition}

\begin{proposition}
  Let $(L, R)$ be an awfs on a category $\catc$, and $\mathcal{A}$ a
  small category. Suppose we are given a morphism $f$ in the functor
  category $\catc^\mathcal{A}$.
  Note that we can view $f$ as a functor
  $f \colon \mathcal{A} \to \catc^\btwo$. Then $R_\mathcal{A}$-algebra
  structures on $f$ naturally correspond to functors $\alpha \colon
  \mathcal{A} \to \rmap$ in the following commutative diagram.
  \begin{equation*}
    \xymatrix{ & \rmap \ar[d] \\
      \mathcal{A} \ar[ur]^\alpha \ar[r]_f & \catc^\btwo}
  \end{equation*}

  Dually for $L_\mathcal{A}$-coalgebra structures.
\end{proposition}

\begin{proof}
  This is straightforward to check.
\end{proof}

\begin{definition}
  \label{def:pointwiseams}
  Given an ams $\xi \colon (C, F^t) \to (C^t, F)$ with structured weak
  equivalences $\wmap$ and a small category $\mathcal{A}$, we define
  the pointwise ams with structured weak equivalences as follows. We
  define $(C_\mathcal{A}, F^t_\mathcal{A})$ and $(C^t_\mathcal{A},
  F_\mathcal{A})$ to be be the pointwise awfs's. The comparison map is
  also defined pointwise.

  We define the category $\wmap_\mathcal{A}$ to be the functor
  category $\wmap^\mathcal{A}$. To define the functor
  $\wmap_\mathcal{A} \to (\catc^\mathcal{A})^\btwo$, we can instead
  define a functor
  $\mathcal{A} \times \wmap_\mathcal{A} \to \catc^\btwo$. We take this
  to be evaluation followed by the map $\wmap \to \catc^\btwo$.
\end{definition}

\begin{proposition}
  We can define the necessary functors to make definition
  \ref{def:pointwiseams} an ams with structured weak equivalences.
\end{proposition}

\begin{proof}
  These are once again defined pointwise. For illustration, we
  just consider the functor
  $\cmap_\mathcal{A} \times_{\catc^{\mathcal{A} \times \btwo}}
  \wmap_\mathcal{A} \to \rtmap_\mathcal{A}$. Note that this amounts to
  constructing a functor
  $\mathcal{A} \times \cmap^\mathcal{A} \times_{\catc^{\mathcal{A}
      \times \btwo}} \wmap^\mathcal{A} \to \rtmap$. However, the
  evaluation maps give us a functor
  $\mathcal{A} \times \cmap^\mathcal{A} \times_{\catc^{\mathcal{A}
      \times \btwo}} \wmap^\mathcal{A} \to \cmap
  \times_{\catc^{\btwo}} \wmap$. We can then compose this with the map
  we are given from $\cmap \times_{\catc^{\btwo}} \wmap$ to $\rtmap$
  to get the required structure.
\end{proof}

\subsubsection{An Explicit Definition for $\wmap$}
\label{sec:an-expl-defin}

We now show that without loss of generality we can assume $\wmap$ is
given by the following explicit definition\footnote{In an earlier
  draft of this paper this was taken as the definition of $\wmap$.}.

\begin{proposition}
  \label{prop:olddefswe}
  If we are given an ams with structured weak equivalences on a
  category $\catc$, then there is another ams with structured weak
  equivalences, with the same underlying model structure, such that
  $\wmap$ is defined as the pullback below.
  
  \begin{equation}
    \label{eq:55}
    \begin{gathered}
      \xymatrix{ \wmap \ar[r] \ar[d] \pullbackcorner & \rtmap \ar[d] \\
        \catc^\btwo \ar[r]_F & \catc^\btwo}
    \end{gathered}
  \end{equation}
\end{proposition}

\begin{proof}
  Assume that we are given an ams with structured weak equivalences
  given by $\wmap$. Write $\wmap'$ for the category defined as in
  \eqref{eq:55}. Note that to show $\wmap'$ also
  gives structured weak equivalences, it suffices to show that we can
  construct functors from $\wmap'$ to $\wmap$ and from $\wmap$ to
  $\wmap'$ over $\catc^\btwo$.

  We first construct the functor $\wmap' \to \wmap$. Suppose we are
  given a map $f$ and an $F^t$-algebra structure on $F f$. We can use
  this to assign $F f$ the structure of a $\wmap$-weak
  equivalence. Since $C^t f$ is given the structure of a
  $C^t$-coalgebra using the comultiplication map, we can also assign
  it a $\wmap$-weak equivalence structure. Finally, we use the
  functorial $3$-for-$2$ operator to assign a $\wmap$-weak equivalence
  structure to $f$, the composition of $C^t f$ and $F f$.

  We now construct the functor $\wmap \to \wmap'$. Suppose we are
  given a map $f$ with $\wmap$-weak equivalence structure. We factor
  $f$ as $C^t f$ followed by $F f$, using the awfs $(C^t, F)$. We then
  also have a $\wmap$-weak equivalence structure on $C^t f$. Hence we
  can use the functorial $3$-for-$2$ operator to assign $F f$ a
  $\wmap$-weak equivalence structure. We then have both a $F$-algebra
  structure and a $\wmap$-weak equivalence structure on $F f$, which
  gives us an $F^t$-algebra structure on $F f$. But we have now given
  $f$ the structure of a $\wmap'$-structured weak equivalence.
\end{proof}

Note that dually, we could also choose $\wmap$ to be defined by the
pullback below.
\begin{equation*}
  \begin{gathered}
    \xymatrix{ \wmap \ar[r] \ar[d] \pullbackcorner & \ctmap \ar[d] \\
      \catc^\btwo \ar[r]_C & \catc^\btwo}
  \end{gathered}
\end{equation*}

Finally, we note that in this case we can drop one of the required
functions from the definition of ams with structured weak
equivalences.

\begin{proposition}
  Suppose we are given a pre-ams
  $\xi \colon (C^t, F) \rightarrow (C, F^t)$, that $\wmap$ is the
  category defined as in proposition \ref{prop:olddefswe} and
  $(C, F^t)$ is cofibrantly generated. Then we have in any case a
  functor $\fmap \times_{\catc^\btwo} \wmap \to \rtmap$.
\end{proposition}

\begin{proof}
  Note that an $F$-algebra structure on a map $f$ witnesses it as the
  retract of $F f$. Hence, if we are given an $F^t$-algebra structure
  on $F f$, we can assign $f$ also the structure of an
  $F^t$-algebra. Since we defined this using a retract, in general it
  might be only a pointed endofunctor algebra. However, since $(C,
  F^t)$ is cofibrantly generated, we can correct this to obtain an
  algebra structure over the monad.
\end{proof}

\subsection{Constructing Very Good Path Objects from Path Objects}
\label{sec:strong-path-objects}

We now use an ams with structured weak equivalences to define a weaker
version of very good path objects (definition \ref{def:strpathobj})
where we replace the requirement of $r_f$ being a trivial cofibration
to just a weak equivalence. We will then show how to use this to
construct very good path objects.

\begin{definition}
  \begin{enumerate}
  \item A \emph{choice of  path objects} consists of an
    assignment to every $F$-map $f : X \rightarrow Y$ a factorisation
    \begin{equation}
      \label{eq:7b}
      X \stackrel{r_f}{\rightarrow} P(f) \stackrel{p_f}{\rightarrow} X
      \times_Y X
    \end{equation}
    of the diagonal $\Delta : X \rightarrow X \times_Y X$ together
    with a weak equivalence structure on $r_f$ and an $F$-algebra
    structure on $p_f$.
  \item A choice of  path objects is \emph{functorial} if
    the assignment of \eqref{eq:7b} provides the action of objects of a
    functor $\fmap \rightarrow \fmap \times_{\catc} \wmap$.
  \item A choice of  path objects is \emph{stable} when
    every map of $F$-algebras whose underlying square is a pullback
    makes the following square given by functoriality a pullback
    \begin{equation*}
      \begin{gathered}
        \xymatrix{P(f) \ar[r]^{P(h,k)} \ar[d]_{p_f} & P(f')
          \ar[d]_{p_{f'}} \\
          X \times_Y X \ar[r] & X' \times_{Y'} X'}
      \end{gathered}
    \end{equation*}
  \end{enumerate}
\end{definition}

\begin{definition}
  We say an awfs $(L, R)$ is \emph{pullback stable}, if for every
  pullback square of the form below,
  \begin{equation*}
    \xymatrix{ U \pullbackcorner \ar[r] \ar[d]_f & X \ar[d]^g \\
      V \ar[r] & Y}
  \end{equation*}
  the square below given by functoriality is also a pullback.
  \begin{equation*}
    \xymatrix{ \cdot \ar[r] \ar[d]_{R f} & \cdot
      \ar[d]^{R g} \\
      V \ar[r] & Y}    
  \end{equation*}
\end{definition}

\begin{theorem}
  \label{thm:strfromwk}
  Suppose we are given an ams with structured weak equivalences $\xi
  \colon (C^t, F) \to (C, F^t)$ where $(C, F^t)$ is pullback stable.

  Then given a stable functorial choice of path objects we can
  construct a stable functorial choice of very good path objects.
\end{theorem}

\begin{proof}
  Suppose that we are given a choice of  path objects. That is, we
  are given for each fibration $f$ a factorisation
  \begin{equation*}
    \begin{gathered}
      \xymatrix{ X \ar[rr]^\Delta \ar[dr]_{r_f} & & X \times_Y X \\
        & P f \ar[ur]_{p_f} & }
    \end{gathered}
  \end{equation*}
  together with weak equivalence on $r_f$ and $R$-algebra structure on
  $p$. Then we apply the $(C, F^t)$ factorisation to $r_f$ to extend the
  diagram as follows.
  \begin{equation*}
    \begin{gathered}
      \xymatrix{ X \ar[rr]^\Delta \ar[dr]_{r_f} \ar[d]_{C r_f} & & X \times_Y X \\
        M r_f \ar[r]_{F^t r_f} & P f \ar[ur]_{p_f} & }
    \end{gathered}
  \end{equation*}
  Then we may use the functorial 3-for-2 operator to construct a weak
  equivalence structure on $C r_f$ from the weak equivalence
  structures on $r_f$ and $F^t r_f$. Since $C r$ is a cofibration, we
  can construct from this a $C^t$-coalgebra structure on $C r$.
  Furthermore we
  can produce an $R$-algebra structure on $F^t r$ using the comparison
  map, and then assign an $F$-algebra structure to $p_f \circ F^t r_f$
  by composing the structures on $F^t r_f$ and $p_f$.

  All of the above constructions can be done functorially, so the
  following is a functorial choice of very good path objects.
  \begin{equation*}
    \begin{gathered}
      \xymatrix{ X \ar[rr]^\Delta \ar[dr]_{C r_f} & & X \times_Y X \\
        & M r_f \ar[ur]_{p_f \circ F^t r_f} & }
    \end{gathered}
  \end{equation*}

  Now we just need to check that this construction satisfies
  stability.

  Assume that we are given a pullback square as below.
  \begin{equation}
    \label{eq:12}
    \begin{gathered}
      \xymatrix{ X \pullbackcorner \ar[r] \ar[d]^f & X' \ar[d]^{f'} \\
        Y \ar[r] & Y'}
    \end{gathered}
  \end{equation}

  We first check that the square below is a pullback.
  \begin{equation}
    \label{eq:29}
    \begin{gathered}
      \xymatrix{ X \times_Y X \ar[r] \ar[d] & X' \times_{Y'} X' \ar[d] \\
        Y \ar[r] & Y'}
    \end{gathered}
  \end{equation}
  To this end, consider the following commutative cube.
  \begin{equation*}
    \begin{gathered}
      \xymatrix@=1.3em{ & X \times_Y X \ar[rr] \ar '[d] [dd] \ar[dl] & & X'
        \times_{Y'} X' \ar[dl] \ar[dd] \\
        X \ar[rr] \ar[dd] & & X' \ar[dd] & \\
        & X \ar'[r][rr] \ar[dl] & & X' \ar[dl] \\
        Y \ar[rr] & & Y' & }
    \end{gathered}
  \end{equation*}
  The front face is a pullback by our assumption and the left and
  right faces are pullbacks by definition. Hence the back face is also
  a pullback. But the bottom face is a pullback once again by our
  assumption. Hence the composition of the back face and bottom face
  is a pullback, but this is precisely \eqref{eq:29}, as required.

  Now consider the diagram
  \begin{equation}
    \label{eq:31}
    \begin{gathered}
      \xymatrix{ X \ar[r] \ar[d] & X' \ar[d] \\
        P f \ar[r] \ar[d] & P f' \ar[d] \\
        X \times_Y X \ar[r] \ar[d] & X' \times_{Y'} X' \ar[d] \\
        Y \ar[r] & Y'}
    \end{gathered}
  \end{equation}
  We have just checked that the lower square is a pullback. The middle
  square is also a pullback since we assumed $P$ is stable. The entire
  rectangle is also a pullback by assumption (it is precisely
  \eqref{eq:12}). We deduce that the upper square is also a pullback.
  
  Finally consider the following diagram.
  \begin{equation}
    \label{eq:32}
    \begin{gathered}
      \xymatrix{ M r_f \ar[r] \ar[d] \ar[d]^{R^t r_f} & M r_{f'}
        \ar[d]^{R^t r_{f'}} \\
        P f \ar[r] \ar[d]^{p_f} \ar[r] & P f' \ar[d]^{p_{f'}} \\
        X \times_Y X \ar[r] & X' \times_{Y'} X'}
    \end{gathered}
  \end{equation}
  Since the upper square in \eqref{eq:31} is a pullback and $(C, F^t)$
  preserves pullbacks, we deduce that the upper square in
  \eqref{eq:32} is a pullback. The lower square is also a pullback by
  the assumption that $P$ is stable. Hence the entire rectangle is a
  pullback. But this is precisely what we need to show that the very good
  path objects we defined before are stable.
\end{proof}

\section{Some Sufficient Conditions for the Existence of Identity
  Types}
\label{sec:some-suff-cond}

In this section we give a direct proof that identity types can be
constructed in certain categories. We will follow the same
construction as in section \ref{sec:strong-path-objects}. Instead of
an ams, we only work with a pre-ams satisfying certain axioms, based
on those considered by Gambino and Sattler in
\cite{gambinosattlerpi}. These axioms are much easier to show for the
examples we consider than the construction of a $3$-for-$2$ operator.

We note that throughout this section, part of the work lies in adapting
arguments based on wfs's to stronger functorial versions. Under
suitable conditions it is possible to save work by doing this
automatically using pointwise awfs's or (as the author will show in a
future paper) category indexed family fibrations. However, for now we
work directly with the functorial versions, to give a clearer picture
of the objects and maps involved in the construction. Since all the
constructions involved are fairly simple, this does not cause too much
difficulty.

We first review some necessary background material.

\subsection{The Leibniz Construction}
\label{sec:leibniz-construction}

The Leibniz construction is a well known construction in homotopical
algebra. See e.g. \cite[Construction 11.1.7]{riehlcht} for a standard
reference. It was first applied to the semantics of homotopy type
theory, and in particular CCHM cubical sets by Gambino and Sattler in
\cite{gambinosattlerpi}. The idea is that given a monoidal category,
$(\catc, \otimes)$, we can give $\catc^\btwo$ also the structure of a
monoidal category using pushout product $\hat \otimes$. Furthermore,
if we are given right adjoints to $- \otimes X$ for each $X$ in
$\catc$, we can produce also right adjoints to $- \hat \otimes f$ for
each $f$ in $\catc^\btwo$ using pullback hom.

\begin{definition}
  Let $(\catc, \otimes)$ be a monoidal category with pushouts. The
  \emph{pushout product} is the monoidal product $\hat{\otimes}$
  defined on $\catc^\btwo$ as follows. Given $f \colon U \rightarrow
  V$ and $g \colon X \rightarrow Y$, we define $f \hat{\otimes} g$ as
  the map given by the universal property of the pushout below.
  \begin{equation*}
    \xymatrix{ U \otimes X \ar[r]^{f \otimes X} \ar[d]_{U \otimes g} &
      V \otimes X \ar[d] \ar@/^/[ddr]^{V \otimes g} & \\
      U \otimes Y \ar[r] \ar@/_/[drr]_{f \otimes Y}
      & \cdot \pushoutcorner \ar[dr]|{f \hat \otimes g} & \\
      & & V \otimes Y}
  \end{equation*}
\end{definition}

\begin{definition}
  Let $(\catc, \otimes)$ be a monoidal category with pushouts and
  pullbacks. Suppose that for each $X$, $- \otimes X$ has a right
  adjoint $\hom(X, - )$. Then for each
  map $f$, $f \hat \otimes -$ has a right adjoint,
  $\monexpl{f}{-}$ referred to as \emph{pullback hom}, which
  is defined explicitly as the map given by the universal property of
  the pullback below. Let $f \colon U \rightarrow V$ and
  $g \colon X \rightarrow Y$.
  \begin{equation*}
    \xymatrix{ \hom(V, X) \ar[dr]|{\monexpl{f}{g}}
      \ar@/^/[drr]^{\hom(V, g)} \ar@/_/[ddr]_{\hom(f, X)} & & \\
      & \cdot \pullbackcorner \ar[r] \ar[d] & \hom(V, Y) \ar[d]^{\hom(f, Y)} \\
      & \hom(U,X) \ar[r]_{\hom(U, g)} & \hom(U,Y)
    }
  \end{equation*}
\end{definition}

\subsection{The Conditions}
\label{sec:conditions}

\begin{definition}
  A monoidal product $\otimes$ is \emph{affine} if the unit of the
  monoidal product is a terminal object.
\end{definition}

\begin{remark}
  For any affine monoidal product, we can define projection maps $X
  \otimes Y \rightarrow X$ and $X \otimes Y \rightarrow Y$.
\end{remark}

Our basic set up is a monoidal category $(\catc, \otimes)$ satisfying
the following conditions.
\begin{enumerate}
\item $\catc$ is finitely complete and finitely cocomplete.
\item $\otimes$ is an affine and symmetric monoidal product that
  preserves colimits.
\item $\delta_i \colon 1 \rightarrow \intv$ for $i = 0, 1$ is an interval
  object.
\item $- \otimes \intv$ has a right adjoint, which we denote $P$.
\item $(C, F^t)$ is pullback stable.
\item Axioms \ref{ax:intvbdrylpcof}, \ref{ax:fibtoltrivfib} and
  \ref{ax:fibtobdyfib} below.
\end{enumerate}

\begin{axiom}
  \label{ax:intvbdrylpcof}
  If we are given a cofibration $m$, then we can also give
  $m \hat{\otimes} [\delta_0, \delta_1]$ the structure of a
  cofibration, and this assignment is functorial in $m$.
\end{axiom}


Note that pullback hom might not be defined in general, but we do know
that $- \hat \otimes \delta_i$ has a right adjoint given by
$\monexpl{\delta_i}{-}$ using $P$.

\begin{axiom}
  \label{ax:fibtoltrivfib}
  Given an $F$ algebra structure on a map $f$, we can define in a
  functorial way, an $F^t$ algebra structure on
  $\monexpl{\delta_i}{f}$ for $i = 0, 1$.
\end{axiom}

\begin{axiom}
  \label{ax:fibtobdyfib}
  Given an $F$ algebra structure on a map $f$, we can define in a
  functorial way, an $F$ algebra structure on
  $\monexpl{[\delta_0, \delta_1]}{f}$.
\end{axiom}


\subsection{Some Useful Propositions}
\label{sec:some-usef-prop}

Before giving some examples, we prove a couple of propositions that
will be useful for verifying the examples do satisfy the axioms, and
later for the theorem itself.

\begin{proposition}
  \label{prop:fibtoltrivfibdual}
  Suppose that a pre-ams $\xi \colon (C^t, F) \rightarrow (C, F^t)$
  satisfies axiom \ref{ax:fibtoltrivfib}. Then given a cofibration
  $m$, we can assign $m \hat{\otimes} \delta_i$ the structure of a
  trivial cofibration.
\end{proposition}

\begin{proof}
  This follows from the adjunction between pushout product and
  pullback hom.
\end{proof}

As remarked in section \ref{sec:introduction}, when working with
cofibrations, we usually use the category $\cmap$ of copointed
endofunctor coalgebras. However, in order to construct the comparison
map it is useful to instead work over the category $\cMap$ of comonad
coalgebras. Because of this, we will aim towards a lemma constructing
a functor $\cmap \to \cMap$ over $\catc^\btwo$ using the assumption
that $(C, F^t)$ is pullback stable.

\begin{lemma}
  \label{lem:retractsarepbs}
  ``Every retract of a monomorphism is a pullback.'' More formally,
  suppose we are given diagram as below, where $k \circ h = 1_X$ and
  $m \circ l = 1_Y$, and that $g$ (and so also $f$) is a
  monomorphism.
  \begin{equation}
    \label{eq:3}
    \begin{gathered}
      \xymatrix{ X \ar@{>->}[d]_f \ar[r]^h & Z \ar@{>->}[d]^g \ar[r]^k & X
        \ar@{>->}[d]^f \\
        Y \ar[r]_l & W \ar[r]_m & Y}
    \end{gathered}
  \end{equation}

  Then the left hand square in \eqref{eq:3} is a pullback.
\end{lemma}

\begin{proof}
  Suppose we are given a commutative diagram as in the solid lines
  below. We need to show there is a unique map $t$ as in the dotted
  line below, making the diagram commute.
  \begin{equation*}
    \xymatrix{ A \ar@/^/[drr]^p \ar@/_/[ddr]_q \ar@{.>}[dr]|t & & \\
      & X \ar[r]^h \ar@{>->}[d]_f & Z \ar@{>->}[d]^g \\
      & Y \ar[r]_l & W}
  \end{equation*}

  Note however that uniqueness easily follows from the fact that $f$
  is monic. Hence we only have to show existence.

  We take $t$ to be $k \circ p$. It is straightforward to verify that
  the resulting diagram commutes.
\end{proof}

\begin{proposition}
  \label{prop:fixcofibs}
  Suppose that $(C, F^t)$ is pullback stable and that every
  cofibration is a monomorphism. Then there is a functor
  that takes a copointed endofunctor coalgebra structure on a map $f$,
  and returns a comonad coalgebra structure on the same map. That is,
  we construct a functor $\cmap \to \cMap$ over $\catc^\btwo$.
\end{proposition}

\begin{proof}
  Suppose that $f$ has the structure of a coalgebra over the
  underlying copointed endofunctor of $(C, F^t)$. The coalgebra
  structure witnesses $f$ as a retract of $C f$. Hence by lemma
  \ref{lem:retractsarepbs} the same diagram witnesses $f$ as a
  pullback of $C f$. Since $(C, F^t)$ is pullback stable we can
  pullback the $C$-coalgebra structure on $C f$ to obtain a
  $C$-coalgebra structure on $f$.

  This is clearly functorial.
\end{proof}

\subsection{Examples}
\label{sec:examples}

\subsubsection{Gambino-Sattler Axioms}
\label{sec:gamb-sattl-axioms}

As stated above, these axioms are based on those of Gambino and
Sattler in \cite[Section 7]{gambinosattlerpi}.

We can recover a similar definition to Gambino and Sattler's as
follows. We add the requirement that $(C, F^t)$ is algebraically
cofibrantly generated, by a diagram $M \colon J \to \catc^\btwo$ that
we refer to as the \emph{generating cofibrations}. We then further
require that $(C^t, F)$ is cofibrantly generated by the coproduct of
the two diagrams of the form $M \poprod \delta_i$ for $i =
0,1$. Instead of assuming axioms \ref{ax:fibtoltrivfib} and
\ref{ax:fibtobdyfib}. We will keep assuming that axiom
\ref{ax:intvbdrylpcof} holds.

In this set up, the awfs $(C^t, F)$ is by definition cofibrantly
generated by the pushout product of generating cofibrations with the
endpoint inclusions $\delta_i$. Observe that using the adjunction
between pushout product and pullback exponential, this automatically
gives us axiom \ref{ax:fibtoltrivfib}. Furthermore, we also get axiom
\ref{ax:fibtobdyfib}. We again use the adjunction between pushout
product and pullback exponential to instead check the dual definition
for generating trivial cofibrations. Any generating trivial
cofibration is of the form $m \hat \otimes \delta_i$ where $m$ is a
generating cofibration. Then by symmetry of the monoidal product, we
have the following isomorphism.
\begin{equation}
  \label{eq:4}
  (m \poprod \delta_i) \poprod [\delta_0, \delta_1]
  \cong (m \poprod [\delta_0, \delta_1]) \poprod \delta_i
\end{equation}
Combining this with axiom
\ref{ax:intvbdrylpcof}, for each generating trivial cofibration $m
\poprod \delta_i$, we can assign $(m \poprod \delta_i) \poprod
[\delta_0, \delta_1]$ with the structure of a trivial
cofibration. Moreover, noting that \eqref{eq:4} is part of a natural
isomorphism and using the functorial part of \ref{ax:intvbdrylpcof},
we can ensure that a generating morphism of trivial cofibrations is
sent to a morphism of trivial cofibrations. This then gives us axiom
\ref{ax:fibtobdyfib}.

If we follow Gambino and Sattler in assuming that the generating
cofibrations are functorially closed under pushout product with
endpoint inclusion, one can construct the comparison map of the
pre-ams using Riehl's observation in \cite[Remark 3.6]{riehlams} that
it suffices to show that one can functorially assign the generating
trivial cofibrations with cofibration structures, and also using
proposition \ref{prop:fixcofibs}.

There are, however still a few minor differences with the
Gambino-Sattler axioms. We work with symmetric monoidal products and an
interval rather than functorial cylinders. Axiom
\ref{ax:intvbdrylpcof} does not seem to follow from \cite[Definition
7.1]{gambinosattlerpi}. Note however that it does hold in the examples
of simplicial sets and (CCHM) cubical sets appearing in \cite[Example
7.2]{gambinosattlerpi}. More generally, using the axiom that
$- \hat \otimes \delta_i$ preserves cofibrations, it follows from the
additional assumption that cofibrations are functorially closed under
binary union. This assumption was added, for instance by Sattler in
\cite[Definition 3.2]{sattlermodelstructures}, and also appears in the
Orton-Pitts axioms \cite{pittsortoncubtopos}.

\subsubsection{BCH Cubical Sets and $01$-Substitution Sets}
\label{sec:bch-cubical-sets}

BCH cubical sets were introduced by Bezem, Coquand and Huber in
\cite{bchcubicalsets}, and further developed by Huber in
\cite{huberthesis}. They were later still further developed by Bezem,
Coquand and Huber in \cite{bchunivalence}, where they showed how to
interpret the univalence axiom in the model. In \cite{pittsnompcs}
Pitts showed that this category of cubical sets is equivalent to a
category based on nominal sets, denoted $01$-substitution sets. Since
this paper is only concerned with structure that is preserved up to
isomorphism by equivalences of categories, we can freely switch back
and forth between the two presentations.

The monoidal product we use is \emph{separated product}. This is
naturally defined in $01$-substitution sets, where it has the same
definition as the separated product in nominal sets. The right adjoint
to $- \otimes \intv$ exists and can be explicitly described in terms
of name abstraction. See \cite[Section 2.4]{huberthesis} for a
detailed description by Huber.

In \cite[Section 7.5.3]{swanliftprob} the author showed that trivial
fibrations can be viewed as cofibrantly generated in two senses. They
can be viewed as cofibrantly generated in Garner's sense by boundary
inclusions together with a uniformity condition (which in loc. cit. is
referred to as \emph{cofibrantly generated with respect to the
  category indexed families fibration}). This is essentially the same
as the definition by Bezem, Coquand and Huber in
\cite{bchunivalence}. Trivial fibrations can also be viewed as
cofibrantly generated with respect to the codomain fibration. The
latter description can be used to show that the awfs $(C, F^t)$ of
cofibrations and trivial fibrations is stable under pullback if it
exists. In fact, using \cite[Corollary 7.5.5]{swanliftprob} we can
show that $(C, F^t)$ is \emph{strongly fibred} with respect to the
codomain fibration. This means that
it forms part of a fibred
awfs over the codomain fibration in which the restriction to each
fibre category is pullback stable.
This leaves the problem of actually constructing
the awfs $(C, F^t)$, which can be done using \cite[Theorem
6.14]{swanwtypered}, together with the observation that the generating
cofibrations are locally decidable, or by using Garner's small object
argument \cite{garnersmallobject}. The awfs can also be constructed
directly in a similar manner to \cite{swannomawfs},
followed by direct verification of pullback stability. See
\cite{bchunivalence} for such a construction (or
\cite{swanidamsdraft}). Note that $n$-dimensional boundary inclusions
can be defined as the pushout product of $n$ copies of
$[\delta_0, \delta_1]$. It follows that the generating cofibrations
are closed under pushout product with $[\delta_0, \delta_1]$, which
then ensures that axiom \ref{ax:intvbdrylpcof} is satisfied.

We define the interval object to be the same as in \cite[Section
6.1]{bchcubicalsets}.

We define the awfs $(C^t, F)$ of trivial cofibrations and fibrations
to be cofibrantly generated by pushout product of a generating
cofibration with an endpoint inclusion $\delta_i$, following the same
construction as in section \ref{sec:gamb-sattl-axioms}. Note that the
pushout product of a boundary inclusion with $\delta_i$ gives us the
standard open box in direction $i$. Hence this gives the same
definition of fibration as given by Huber in \cite[Remark
3.9]{huberthesis} or by the author in \cite[Section
5]{swannomawfs}. This defines the awfs $(C^t, F)$ uniquely up to
isomorphism. In order to show that $(C^t, F)$ actually exists, one can
either apply Garner's small object argument, or give a direct
description as in \cite{swannomawfs}. Once again this definition
automatically gives us axioms \ref{ax:fibtoltrivfib} and
\ref{ax:fibtobdyfib}.

This time constructing the comparison map is a little
trickier. However, it can be done using the following lemma.

\begin{lemma}
  \label{lem:gentrivcofiscof}
  Suppose that we are given a finitely cocomplete affine monoidal
  category $(\catc, \otimes)$ together with an awfs $(L, R)$ and an
  interval object $\delta_0, \delta_1 \colon 1 \rightarrow \intv$.

  If we are given an $L$-coalgebra structure on a map $m$ and an
  $L$-coalgebra structure on $[\delta_0, \delta_1] \poprod m$,
  then we can produce an $L$-coalgebra structure on $\delta_i
  \poprod m$ for $i = 0, 1$.

  Moreover, this assignment is functorial, in the sense that it is the
  action on objects of a functor $\lMap \times \lMap \to \lMap$.
\end{lemma}

\begin{proof}
  We will just do the case $i = 0$, the other case being
  similar. Roughly the idea is that if we are given an $n$ dimensional
  open box, we can find a filler by first building the top lid (which
  is $n - 1$ dimensional), and then filling the resulting boundary of
  the $n$-cube. We will show how $\delta_0 \poprod m$ can be built up
  from cofibrations.

  Say that $m \colon A \rightarrow B$ and write $D$ for the domain of
  $\delta_0 \poprod m$. By definition, $D$ is given by the
  following pushout diagram.
  \begin{equation*}
    \xymatrix{ 1 \otimes A \ar[r]^{\delta_0 \otimes A} \ar[d]_{1 \otimes
        m} &
      \intv \otimes A \ar[d] \\
      1 \otimes B \ar[r] & D \pushoutcorner}
  \end{equation*}

  Now if $L$ is the domain of $[\delta_0, \delta_1] \poprod m$,
  then note that we have a canonical map $i \colon D \rightarrow C$
  given by the inclusion $\iota_1 \colon 1 \rightarrow 2$. This gives
  us a commutative square of the form below.
  \begin{equation*}
    \xymatrix{ 1 \otimes A \ar[r]^{\iota_1} \ar[d]_{1 \otimes m} & \intv
      \otimes A \ar[r] & D \ar[d] \\
      1 \otimes B \ar[r]_{\iota_1 \otimes B} & 2 \otimes B \ar[r] & C}
  \end{equation*}
  
  One can verify by diagram chase that in fact this square is a
  pushout. We think of this as gluing the missing lid to the
  $n$-dimensional open box to make it into the boundary of the
  $n$-dimensional cube. Further diagram chasing shows that in fact
  $\delta_0 \poprod m$ factors as the map $D \rightarrow C$
  followed by $[\delta_0, \delta_1] \poprod m$. Informally, we
  visualise this as including the $n$-dimension open box into the
  $n$-cube, by first including it into the boundary, and then
  including the boundary into the $n$-cube.

  However, we have now exhibited $\delta_0 \poprod m$ as a pushout
  of a generating cofibration, $m$ composed with
  $[\delta_0, \delta_1] \poprod m$, which is also a generating
  cofibration. Both pushout and composition preserve $L$-coalgebra
  structure, so we obtain an $L$-coalgebra structure on
  $\delta_0 \poprod m$.

  Functoriality is tedious but straightforward to verify.
\end{proof}

\begin{theorem}
  Suppose that $M \colon J \to \sub^\btwo$ is the generating diagram of
  cofibrations, defined as above, using boundary inclusions. Then
  $\delta_0 \poprod M$ and $\delta_1 \poprod M$ factor through the
  forgetful functor from $C$-coalgebras to $\sub^\btwo$.
\end{theorem}

\begin{proof}
  As explained above, we can show how to functorially assign
  $[\delta_0, \delta_1] \poprod M$ with the structure of a
  $C$-coalgebra. However, we can now apply lemma
  \ref{lem:gentrivcofiscof}.
\end{proof}

In this setting it is also possible to view axioms
\ref{ax:intvbdrylpcof}, \ref{ax:fibtoltrivfib} and
\ref{ax:fibtobdyfib} in a more geometric fashion. Axiom
\ref{ax:intvbdrylpcof} corresponds to the fact that given a boundary
of an $n$ dimensional cube, we obtain a new boundary of an $n + 1$
dimensional cube by taking the product with the interval to get a
tube, and then pasting on both ends of the tube. As stated above,
axiom \ref{ax:fibtoltrivfib} follows from the way we define generating
trivial cofibrations, in which a generating trivial cofibration is an
open box, which can be seen as a tube with only one of the ends pasted
on. Finally axiom \ref{ax:intvbdrylpcof} corresponds to the fact that
if we are given an open box we obtain a new open box when we take the
product with the interval, and then paste on both sides of the
prism, while leaving the top open.

\subsubsection{Van den Berg-Frumin/Orton-Pitts Axioms}
\label{sec:van-den-berg-1}

In \cite[Section 7.5.2]{swanliftprob} the author gave a definition
related to a class of structures considered by Orton and Pitts in
\cite{pittsortoncubtopos} and by Van den Berg and Frumin in
\cite{vdbergfrumin}. This presentation is closest to that of Van den
Berg and Frumin, which in turn is based on the Gambino-Sattler
definition above.
We work over a locally cartesian closed category
$\catc$ with finite colimits, disjoint coproducts, and an interval object
$\delta_0, \delta_1 \colon 1 \to \intv$. In order to construct the
awfs's $(C, F^t)$ and $(C^t, F)$ we will assume that $\catc$ satisfies
one of the ``codomain fibred'' versions of the small object argument
developed by the author in \cite{swanwtypered}. For example, it
suffices that $\catc$ is a topos with natural number object and
satisfies WISC. Examples of such structures include CCHM cubical sets,
as defined in \cite{coquandcubicaltt} and simplicial sets. For now, we
don't assume that the interval object has connections, although we
will see in section \ref{sec:using-conn-homot} that in that case we
can simplify the argument.

In this setup we start with an awfs $(C, F^t)$ which is cofibrantly
generated with respect to the codomain fibration by a family of maps
of the form below.
\begin{equation*}
  \xymatrix{ 1 \ar[rr]^\top \ar[dr]_\top & & \Sigma \ar@{=}[dl] \\
    & \Sigma & }
\end{equation*}
In this case we automatically get pullback stability, as for BCH
cubical sets. For
now we don't assume that the generating cofibrations are closed under
composition, although we'll see later in section
\ref{sec:using-step-one} that in this case the argument simplifies. We
do however, assume that cofibrations are closed under finite
union and that both endpoint inclusions $\delta_i$ are
cofibrations. This ensures we get axiom \ref{ax:intvbdrylpcof}.

We again define $(C^t, F)$ as cofibrantly generated by pushout
product with a cofibration and an endpoint inclusion. In contrast to
BCH cubical sets, in this case $(C^t, F)$ will also be cofibrantly
generated with respect to the codomain fibration on $\catc$. For this
to work smoothly, instead of working with arbitrary monoidal products
as before, we only consider cartesian product, which easily extends to
a fibred monoidal product over the codomain fibration, which then
ensures pushout product is also fibred \cite[Section 6]{swanliftprob}.
We then define
$(C^t, F)$ to be the awfs cofibrantly generated by the coproduct of
the following two families of maps.
\begin{equation*}
  \begin{gathered}
    \xymatrix{ \intv +_1 \Sigma \ar[rr]^{\delta_0 \hat \times \top} \ar[dr]
      & & \intv
      \times \Sigma \ar[dl] \\
      & \Sigma & }
  \end{gathered}
  \qquad
  \begin{gathered}
    \xymatrix{ \intv +_1 \Sigma \ar[rr]^{\delta_1 \hat \times \top} \ar[dr]
      & & \intv
      \times \Sigma \ar[dl] \\
      & \Sigma & }
  \end{gathered}
\end{equation*}
See \cite[Section 7.5.2]{swanliftprob} for more details. Note that we
again easily obtain axioms \ref{ax:fibtoltrivfib} and
\ref{ax:fibtobdyfib} from this definition.

\subsection{Proof of Existence of Identity Types}
\label{sec:proof-exist-ident}

First note that we can define a functorial choice of factorisations of
diagonal maps using the following well known construction, usually
referred to as \emph{mapping path space}.

Given a map $f \colon X \to Y$, we define $P(f)$ to be given by the
pullback below, where the bottom map $Y \to P(Y)$ corresponds to the
projection $Y \otimes \intv \to Y$ under the adjunction.
\begin{equation*}
  \xymatrix{ P(f) \pullbackcorner \ar[r] \ar[d] & P(X) \ar[d] \\
    Y \ar[r] & P(Y)}
\end{equation*}
When $f$ is clear from the context, we will also write $P(f)$ as
$P_Y(X)$.

Note that we have evident maps $r \colon X \to P_Y(X)$ and projections
$p_0, p_1 \colon P_Y(X) \to X$ over $Y$.

One can verify by diagram chase that
$p_f := \langle p_0, p_1 \rangle \colon P_Y(X) \to X \times_Y X$ can
be viewed as a pullback of $\monexpl{f}{[\delta_0, \delta_1]}$ (see
\cite[Proposition 2.3.3]{voevodskykapulkinlumsdainess} for the
analogous statement in simplicial sets). We can
therefore assign it the structure of a fibration by first applying
axiom \ref{ax:fibtobdyfib} to give $\monexpl{f}{[\delta_0, \delta_1]}$
the structure of a fibration, and then assigning $p_f$ the unique
fibration structure preserved by the pullback.

We construct the choice of factorisations $\id_Y(X)$ using $(C, F^t)$
in the pre-ams together with $P$, exactly like in theorem
\ref{thm:strfromwk}. It remains to show that this does give us a
stable functorial choice of very good path objects.

We will use a structured version of the definitions of homotopy and
strong deformation retract defined below. This is based on the
non-structured version used e.g. by Gambino and Sattler in
\cite[Remark 4.2]{gambinosattlerpi}.

\begin{definition}
  Suppose we are given morphisms $f,g \colon X \to Y$. A
  \emph{structured homotopy} from $f$ to $g$ is a map $h \colon X
  \otimes \intv \to Y$ fitting into the following commutative diagram.
  \begin{equation*}
    \begin{gathered}
      \xymatrix{ X \ar[dr]|{X \otimes \delta_0} \ar@/^/[drr]^{f} & & \\
        & X \otimes \intv \ar[r]^h & Y \\
        X \ar[ur]|{X \otimes \delta_1} \ar@/_/[urr]_{g} & &}
    \end{gathered}
  \end{equation*}

  Suppose we are given maps $f,g \colon X \to Y$,
  $f', g' \colon X' \to Y'$, together with a structured homotopy $h$
  from $f$ to $g$ and a structured homotopy $h'$ from $f'$ to $g'$,
  and maps $k \colon X \to X'$ and $l \colon Y \to Y'$. We say $k$ and
  $l$ \emph{preserve the structured homotopies} if the following
  equations hold.
  \begin{enumerate}
  \item $l \circ f = f' \circ k$
  \item $l \circ g = g' \circ k$
  \item $h' \circ (k \otimes \intv) = l \circ h$
  \end{enumerate}
\end{definition}

\begin{definition}
  A \emph{(structured) strong deformation retract} from $X$ to $Y$
  consists of the following.
  \begin{enumerate}
  \item A morphism $f \colon X \to Y$.
  \item A morphism $s \colon Y \to X$.
  \item A structured homotopy $h$ from $s \circ f$ to $1_X$.
  \end{enumerate}
  The maps are required to satisfy the following equalities.
  \begin{enumerate}
  \item $f \circ s = 1_Y$
  \item $h \circ (s \otimes \intv) = s \circ \pi_0$ (where $\pi_0$ is
    the projection $X \otimes \intv \to X$)
  \end{enumerate}

  Suppose we are given strong deformation retracts $(f, s, h)$ and
  $(f', s', h')$, together with a map $k \colon X \to X'$ and
  $l \colon Y \to Y'$. We say $k$ and $l$ \emph{preserve the
    structured strong deformation retracts} if they satisfy the
  following equalities.
  \begin{enumerate}
  \item $l \circ f = f' \circ k$
  \item $k \circ s = s' \circ l$
  \item $k$ and $l$ preserve the structured homotopies $h$ and $h'$
  \end{enumerate}
\end{definition}

\begin{lemma}
  \label{lem:projistrivfib}
  Suppose we are given a fibration $f \colon X \rightarrow Y$. Then we
  can given each projection $e_i \colon P_Y X \rightarrow X$ the
  structure of a trivial fibration. Moreover this is functorial, in
  the sense that given a morphism of fibrations, the commutative
  square derived from the functoriality of $P_Y X$ is a morphism of
  trivial fibrations.
\end{lemma}

\begin{proof}
  It is straightforward to exhibit $e_i$ as a pullback of
  $\monexpl{\delta_i}{f}$, which has the structure of a trivial
  fibration by axiom \ref{ax:fibtoltrivfib}. We assign $e_i$ the
  unique trivial fibration structure that makes the pullback a
  morphism of trivial fibrations.
\end{proof}

\begin{lemma}
  \label{lem:mkstrongdefretract}
  Suppose we are given maps $r \colon A \rightarrow B$ and
  $i \colon B \rightarrow A$ such that $i \circ r = 1_A$,
  together with the structure of a
  cofibration on $r$ and the structure of a trivial fibration on
  $i$. Then we can extend $r$ and $i$ to a structured strong
  deformation retract.

  Also, this is functorial in the sense that every commutative square that
  preserves the cofibration and trivial fibration structures leads to
  a commutative square preserving the strong deformation retract
  structures.
\end{lemma}

\begin{proof}
  We will define a lifting problem of
  $r \hat{\otimes} [\delta_0, \delta_1]$ against $i$. By axiom
  \ref{ax:intvbdrylpcof}, we will be able to find a diagonal filler.

  The domain of $r \hat{\otimes} [\delta_0, \delta_1]$ is by
  definition the pushout below.
  \begin{equation*}
    \xymatrix{ A \otimes 2 \ar[r]^{A \otimes [\delta_0, \delta_1]}
      \ar[d]_{r \otimes \intv} & A \otimes \intv \ar[d] \\
      B \otimes 2 \ar[r] & \pushoutcorner
      A \otimes \intv +_{A \otimes 2} B \otimes 2}
  \end{equation*}
  To define a map $A \otimes \intv +_{A \otimes 2} B \otimes 2
  \rightarrow B$ is therefore to define a map $B \otimes 2 \rightarrow
  B$ and a map $A \otimes \intv \rightarrow B$ ensuring that the two
  maps $A \otimes 2 \rightarrow B$ agree.
  
  Note that since $\otimes$ is affine and preserves colimits, we have
  $B \otimes 2 \cong B + B$. We can therefore define a map $B \otimes
  2 \rightarrow B$ as $[r \circ i, 1_B]$.

  We define the map $A \otimes \intv \rightarrow B$ to be $r \circ
  \pi_0$.

  This then gives us the following lifting problem, where the top
  morphism is defined as above.
  \begin{equation*}
    \xymatrix{ A \otimes \intv + B \otimes 2 \ar[d]_{r \hat{\otimes}
        [\delta_0, \delta_1]} \ar[r] & B \ar[d]^i \\
      B \otimes \intv \ar[r]_{i \circ \pi_0} & A}
  \end{equation*}
  We take $h \colon B \otimes \intv \rightarrow B$ to be the
  diagonal filler given by the cofibration structure on $r
  \hat{\otimes} [\delta_0, \delta_1]$ (which in turn is given by the
  cofibration structure on $r$ via axiom \ref{ax:intvbdrylpcof})
  together with the trivial fibration structure on $i$.

  We need to check that
  $h \colon B \otimes \intv \rightarrow B$ does witness that $r$ is a
  strong deformation retract. The upper triangle identity ensures that
  $h \,\circ \, B \otimes \delta_0 = r \circ i$, that
  $h \, \circ \, B \otimes \delta_1 = 1_B$, and that
  $h \, \circ \, r \otimes \intv = r \circ \pi_0$.
  Hence, $h$ does
  indeed witness that $r$ is a strong deformation retract.

  Functoriality follows from axiom \ref{ax:intvbdrylpcof} together
  with diagram chasing.
\end{proof}

\begin{lemma}
  \label{lem:reflretract}
  The map $C r \colon X \rightarrow \id_Y(X)$ has the structure of a
  strong deformation retract.

  Furthermore, this is functorial in the following sense. If we are
  given a morphism of fibrations from $f \colon X \rightarrow Y$ to
  $f' \colon X' \rightarrow Y'$, then the commutative square given by
  the functoriality of $\id$ preserves the strong deformation retract
  structure.
\end{lemma}

\begin{proof}
  Note that $C r$ has a retract given by $i := e_0 \circ F^t r$. Note
  that $e_0$ is given the structure of a trivial fibration by lemma
  \ref{lem:projistrivfib}, and $F^t$ also has the structure of a
  trivial fibration. By composing these, we give $i$ also the
  structure of a trivial fibration.

  We now apply lemma \ref{lem:mkstrongdefretract}.
\end{proof}

\begin{lemma}
  \label{lem:cofsdrimpliestrivcof}
  Suppose that we are given a cofibration $t \colon A \rightarrow B$
  with the structure of a strong deformation retract. Then we can
  assign $t$ the structure of a trivial cofibration, and moreover this
  can be done functorially.
\end{lemma}

\begin{proof}
  This is essentially a functorial version of a special case of
  \cite[Lemma 4.3]{gambinosattlerpi}, but for completeness we write
  out the details below.

  We write out the structure of a strong deformation retract as a map
  $i$ such that $i \circ t = 1_A$ and the
  map $h \colon B \otimes \intv \rightarrow B$ in the diagram below.
  \begin{equation}
    \label{eq:1}
    \begin{gathered}
      \xymatrix{ B \ar[dr]|{B \otimes \delta_0} \ar@/^/[drr]^{t \circ i} & & \\
        & B \otimes \intv \ar[r]^h & B \\
        B \ar[ur]|{B \otimes \delta_1} \ar@/_/[urr]_{1_B} & &}      
    \end{gathered}
  \end{equation}
  We also have that $i \circ h = i \circ \pi_0$ and $h \circ t \otimes
  \intv = t \circ \pi_0$.

  We can now exhibit $t$ as a retract of $t \hat{\otimes}
  \delta_0$ in the diagram below. Note that we can show that the right
  hand square really does commute using the upper part of \eqref{eq:1}
  together with the identity $h \circ t \otimes
  \intv = t \circ \pi_0$.
  The upper horizontal composition is trivially the identity.
  We show that the lower horizontal composition is
  the identity on $B$ by using the lower half of \eqref{eq:1}.
  \begin{equation*}
    \xymatrix{ A \ar[d]_t \ar[r]^{A \otimes \delta_1}
      & B +_A A \otimes \intv \ar[d]|{t
        \hat{\otimes} \delta_0} \ar[r]^{[i, \pi_0]} & A \ar[d]^t \\
      B \ar[r]_{B \otimes \delta_1} & B \otimes \intv \ar[r]_{h} &
      B}
  \end{equation*}
  
  However, this allows us to assign $t$ the structure of a trivial
  cofibration from that of $t \hat{\otimes} \delta_0$, which we
  construct using proposition \ref{prop:fibtoltrivfibdual} together
  with axiom \ref{ax:fibtoltrivfib}.
\end{proof}

\begin{theorem}
  The objects $\id_Y(X)$ can be given the structure of a stable
  functorial choice of very good path objects.
\end{theorem}

\begin{proof}
  We need to show how to give the map $C r \colon X \to \id_Y(X)$ the
  structure of a trivial cofibration. We showed in lemma
  \ref{lem:reflretract} how to give it the structure of a
  strong deformation retract. In any case it has the structure of a
  cofibration, and so by lemma \ref{lem:cofsdrimpliestrivcof} we can
  assign it the structure of a trivial cofibration.

  We next show how to assign a fibration structure to the map
  $\id_Y(X) \to X \times_Y X$ defined as the composition of $F^t r
  \colon \id_Y(X) \to P_Y X$
  with the original map $p_f \colon P_Y X \to X \times_Y X$. We know
  that $F^t r$ has the structure of a trivial fibration, so using the
  comparison map $\xi \colon (C^t, F) \to (C, F^t)$ we can assign it
  also the structure of a fibration. In general in an awfs fibrations
  can be composed functorially, giving us a fibration structure on
  $p_f \circ F^t r$, as required.

  Finally, the proof of stability is exactly the same as in theorem
  \ref{thm:strfromwk}, using the assumptions that $P$ is stable and
  that $(C, F^t)$ is pullback stable.
\end{proof}

\section{Two Simplifications in CCHM Cubical Sets}
\label{sec:simpl-when-step}

The first version of this work was largely specific to
$01$-substitution sets, although it was fairly clear that the main
ideas should generalise to other situations. At this time Cohen,
Coquand, Huber and M\"{o}rtberg
were already using the newer definition of cubical sets, which now
appears in \cite{coquandcubicaltt}. There was some discussion between
Thierry Coquand, Simon Huber and the author on how to translate the
ideas into the new definition of cubical sets. During this discussion,
Coquand noticed that in fact in this specific situation, a
simplified definition of identity type can be used, where it
is easier to give explicit definitions of the objects and maps
involved. This simplified version was the one used for the definition
of identity types in \cite{coquandcubicaltt}. The same construction
was applied to a wide class of models by Orton and Pitts in
\cite{pittsortoncubtopos}. Another variant of that construction was
used by Van den Berg and Frumin in \cite{vdbergfrumin} and in fact the
presentation here will be much closer to the Van den Berg-Frumin
version.

\subsection{Using Connections for the Strong Deformation Retract Structure}
\label{sec:using-conn-homot}

The first observation was that for constructing the strong deformation
retract structure one can exploit the fact that CCHM cubical sets
include connections to get a more explicit definition that does not
require the map $f$ to be a fibration.

We will give an explanation of this in the lemmas below. Following
Gambino and Sattler \cite{gambinosattlerpi}, we note that in fact all
we need is that the interval object $\intv$ has connections, in the
form of two maps $c_i \colon I \otimes I \rightarrow I$ for $i = 0, 1$
satisfying appropriate equalities.

\begin{lemma}
  \label{lem:conndefretract}
  Suppose that we are given maps
  $\intv \otimes \intv \rightarrow \intv$ giving connections on
  $\intv$ satisfy the conditions given by Sattler in \cite[Section
  3.2]{sattlermodelstructures}. Then $X \rightarrow P_Y X$ always has
  the structure of a strong deformation retract (even if $f$ is not a
  fibration).
\end{lemma}

\begin{proof}
  As Gambino and Sattler point out in \cite[Section
  2]{gambinosattlerpi}, the connections on $\intv$ give $P$ the
  structure of a functorial cylinder in the opposite category. It
  follows that the maps $X \rightarrow P_Y(X)$ are strong deformation
  retracts by \cite[Remark 4.2]{gambinosattlerpi} noting that the
  extra conditions added in \cite[Section 3.2]{sattlermodelstructures}
  imply that this holds for $P_Y$ for all $Y$ rather than just
  $P$.
\end{proof}

To give the map $X \rightarrow \id_Y(X)$ the structure of a strong
deformation retract, we use the lemma below to ``lift'' the strong
deformation retract structure on the map $X \rightarrow P_Y(X)$.

\begin{lemma}
  \label{lem:liftstrdefretract}
  Suppose we are given a diagram as below.
  \begin{equation*}
    \xymatrix{ & B' \ar[d]^f \\
      A \ar[r]^r \ar[ur]^{r'} & B}
  \end{equation*}
  Suppose further that $r$ is given the structure of a strong
  deformation retract of the form $i \colon B \rightarrow A$ and $h
  \colon B \otimes \intv \rightarrow B$, that $f$ is given the
  structure of a trivial fibration, and that $r'$ is also given the
  structure of a cofibration. Then we can assign $r'$ the
  structure of a strong deformation retract of the form $i' \colon B'
  \rightarrow A$ and $h' \colon B' \otimes \intv \rightarrow \intv$,
  satisfying the following commutative diagrams.
  \begin{equation}
    \label{eq:2}
    \begin{gathered}
      \xymatrix{ & B' \ar[dl]_{i'} \ar[d]^f \\
        A & B \ar[l]^i }
    \end{gathered}
    \qquad
    \begin{gathered}
      \xymatrix{ B' \otimes \intv \ar[r]^{h'} \ar[d]_{f \otimes \intv}
        & B' \ar[d]^f \\
        B \otimes \intv \ar[r]_h & B}
    \end{gathered}
  \end{equation}
  Furthermore this assignment is functorial.
\end{lemma}

\begin{proof}
  First note that the first diagram in \eqref{eq:2} tells us that we
  have to have $i' = i \circ f$, so we take this for the definition of
  $i'$. Then we already see that $i' \circ r' = 1_A$.
  
  Next, following the same outline as in the proof of lemma
  \ref{lem:mkstrongdefretract}, we will define a lifting problem of
  $r' \hat{\otimes} [\delta_0, \delta_1]$ against $f$.

  We again note that to define a map
  $A \otimes \intv +_{A \otimes 2} B' \otimes 2 \rightarrow B'$ is
  therefore to define a map $B' \otimes 2 \rightarrow B'$ and a map
  $A \otimes \intv \rightarrow B'$ ensuring that the two maps
  $A \otimes 2 \rightarrow B'$ agree.

  In fact we define both of these exactly the same as in lemma
  \ref{lem:mkstrongdefretract}.  Namely, we define the map
  $B' \otimes 2 \rightarrow B'$ to be given by $[\delta_0, \delta_1]$
  via the isomorphism $B' \otimes 2 \cong B' + B'$, and we define the
  map $A \otimes \intv \rightarrow B'$ to be $r' \circ \pi_0$.

  We define the lower map $B' \otimes \intv \rightarrow A$ to be $h
  \,\circ\, f \otimes \intv$. We then take $h'$ to be the diagonal
  filler given by the cofibration structure on $r \hat{\otimes}
  [\delta_0, \delta_1]$, as below.
  \begin{equation*}
    \xymatrix{ A \otimes \intv + B' \otimes 2 \ar[d]_{r \hat{\otimes}
        [\delta_0, \delta_1]} \ar[rr] & & B' \ar[d]^{f} \\
      B' \otimes \intv \ar[r]_{f \otimes \intv} \ar@{.>}[urr]^{h'}
      & B \otimes \intv \ar[r]_{h}
      & A}
  \end{equation*}

  Exactly the same as for lemma \ref{lem:mkstrongdefretract}, the
  upper triangle ensures that $h'$ is a homotopy from $r' \circ i'$ to
  $1_{B'}$, and that $h' \,\circ\, r' \otimes \intv = r' \circ \pi_0$.
  Hence, this does give a strong deformation retract.

  The lower triangle tells us that the square in \eqref{eq:2}
  commutes.

  Functoriality is again the same as for lemma
  \ref{lem:mkstrongdefretract}.
\end{proof}

\begin{remark}
  We can in fact recover lemma \ref{lem:mkstrongdefretract} as a
  special case of lemma \ref{lem:liftstrdefretract}, by taking $A = B$
  and $r$ to be $1_A$, which trivially has the structure of a strong
  deformation retract.
\end{remark}

We can now give an alternative proof of lemma
\ref{lem:reflretract}. We split it into two steps. First show that
the reflexivity map $X \rightarrow P_Y(X)$ is a strong deformation
retract using lemma \ref{lem:conndefretract}, and then lift this
structure to the map $X \rightarrow \id_Y(X)$ using lemma
\ref{lem:liftstrdefretract}.

This version of the proof makes essential use of the connections on
$\intv$, but now applies to any map $f$, without needing any fibration
structure on $f$.

\subsection{Avoiding a Transfinite Construction}
\label{sec:using-step-one}

Note that the definition of the cofibrantly generated awfs $(C, F^t)$
according to Garner's small object argument \cite{garnersmallobject}
involves a transfinite construction. The second observation, by Coquand,
was that in fact a transfinite construction is not necessary, and one
can give a much simpler definition that suffices for constructing
identity types.

The key point is that there is a much simpler awfs $(C_1, F^t_1)$ such
that $(C, F^t)$ is algebraically free on the underlying lawfs of
$(C_1, F^t_1)$. One way of understanding $(C_1, F^t_1)$ is as an
internal version of step-one of Garner's small object argument. See
\cite[Section 7.5.2]{swanliftprob} for more precise explanation of
this. As observed by Gambino and Sattler it can also be viewed as a
partial map classifier. See \cite[Remark 9.5]{gambinosattlerpi} for
that description. See also the description by Van den Berg and Frumin in
\cite{vdbergfrumin}.
In any case this construction gives an lawfs, and as
Bourke and Garner show in \cite{bourkegarnerawfs1}, extending an lawfs
to an awfs corresponds precisely to giving a natural way of composing
$C_1$-coalgebras\footnote{Here we mean coalgebras over
  the comonad $C_1$.}. So, the reason that we can do this is that in
CCHM cubical sets the generating cofibrations can be composed
(see \cite[Section 7.5.2]{swanliftprob} for more detail). As
noticed by Sattler (and as explained in \cite[Remark
9.5]{gambinosattlerpi}), although $(C, F^t)$ is algebraically free
on the underlying lawfs of $(C_1, F^t_1)$, these are definitely
different awfs's.

We will now show that in general in this situation we can instead use
$(C_1, F^t_1)$ to construct the identity types.

For convenience, we will continue to assume that we are given a
pre-ams $(C^t, F) \rightarrow (C, F^t)$. However, we observe that now
neither awfs is being used to construct objects, but only to define
the categories of (trivial) cofibrations and fibrations. It is
possible to use this idea to rephrase the results to work without
those awfs's at all. Indeed the proofs in \cite{coquandcubicaltt},
\cite{pittsortoncubtopos} and
\cite{vdbergfrumin} do not use any transfinite construction.

\begin{lemma}
  \label{lem:c1f1tocf}
  Suppose that $(C_1, F^t_1)$ is an lawfs and
  $\zeta \colon (C_1, F^t_1) \rightarrow (C, F^t)$ is a morphism of
  lawfs's witnessing that $(C, F^t)$ is algebraically free on
  $(C_1, F^t_1)$. Suppose further that
  $\mu \colon (F^t_1)^2 \rightarrow F^t$ is a natural transformation
  making $F^t_1$ into a monad. Then for each $f$ we can assign
  $C_1 f$ the structure of a cofibration and $F^t_1 f$ the structure
  of a trivial fibration.

  Moreover, this assignment is functorial in $f$.
\end{lemma}

\begin{proof}
  First note that the morphism $\zeta$ gives us a canonical map from
  $C_1$-coalgebras to $C$-coalgebras commuting with the forgetful
  functors.

  We apply this functor to the canonical $C_1$-coalgebra structure on
  $C_1 f$ given by comultiplication to give it the structure of a
  $C$-coalgebra.

  Similarly, $\zeta$ gives a morphism from $F^t$-algebras to
  $F^t_1$-algebras commuting with the forgetful functor. By the
  definition of algebraic freeness this functor is an isomorphism, and
  so its inverse is a functor from $F^t_1$-algebras (in the pointed
  endofuntor sense) to $F^t$-algebras (in the monad sense).

  Next, note that we can assign $F_1^t f$ the structure of an
  $F_1$-algebra using the multiplication $\mu$. We then forget that
  the algebra respects the multiplication to get an algebra over
  $F_1^t$ as a pointed endofunctor, and then apply the functor above
  to give it the structure of an $F^t$-algebra.
\end{proof}

\begin{theorem}
  Suppose that $\xi \colon (C^t, F) \rightarrow (C, F^t)$ is a pre-ams
  satisfying the conditions in section \ref{sec:some-suff-cond}, 
  that $(C, F^t)$ is algebraically free on an lawfs $(C_1, F^t_1)$ and
  that we are given a multiplication map $\mu$ making $F^t_1$ a
  monad.

  For each $f \colon X \rightarrow Y$, define $\id_Y(X)$ to be given
  by the $(C_1, F^t_1)$ factorisation of the map
  $X \rightarrow P_Y(X)$ as in the diagram below.
  \begin{equation*}
    \xymatrix{ & \id_Y(X) \ar[dr]^{F^t_1 r} & \\
      X \ar[ur]^{C_1 r} \ar[rr]_r & & P_Y(X)}
  \end{equation*}
  Then $\id_Y(X)$ can be equipped with structure of functorial choice
  of very good path objects.
\end{theorem}

\begin{proof}
  We apply lemma \ref{lem:c1f1tocf} to give $C_1 r$ the structure of a
  $C$-coalgebra and $F^t_1 r$ the structure of a $F^t$-algebra. We
  then continue with exactly the same proof as in section
  \ref{sec:proof-exist-ident}
\end{proof}

\bibliographystyle{abbrv}
\bibliography{mybib}{}

\end{document}